\documentclass[12pt]{article}
\usepackage{e-jc}
\usepackage{graphicx}
\usepackage{amsmath,amssymb,amsthm,enumerate}
\usepackage[colorlinks=true,citecolor=black,linkcolor=black,urlcolor=blue]{hyperref}

\MSC{05C50, 15A18, 15B57}

\Copyright{  The authors. Released under the CC BY-ND license (International 4.0).}

\usepackage[mathlines]{lineno}
\numberwithin{figure}{section}

\usepackage{tikz}
\usetikzlibrary{arrows}
\tikzset{
every picture/.style=thick,
every node/.style={circle, draw, inner sep=3pt},
bluenode/.style={circle, draw=black, fill=blue!40, very thick, minimum size=6mm},
whitenode/.style={circle, draw=black, fill=black!10, very thick, minimum size=6mm},
squarednode/.style={rectangle, draw=red!60, fill=red!5, very thick, minimum size=5mm},
every loop/.style={min distance=8mm} 
} 

\theoremstyle{plain}
\newtheorem{thm}{Theorem}[section]
\newtheorem{lem}[thm]{Lemma}
\newtheorem{cor}[thm]{Corollary}
\newtheorem{prop}[thm]{Proposition}

\newtheorem{obs}[thm]{Observation}

\theoremstyle{definition}
\newtheorem{defn}[thm]{Definition}
\newtheorem{ex}[thm]{Example}
\newtheorem{rem}[thm]{Remark}

\newcommand\p{\mathcal{P}}
\newcommand\s{\mathcal{S}}
\newcommand\M{\operatorname{M}}
\newcommand\mr{\operatorname{mr}}
\newcommand\Z{\operatorname{Z}}
\newcommand{\Zsharp}{\Z_{\mathrm{RL}}}
\newcommand{\RL}{\operatorname{RL}}
\newcommand{\rl}{\operatorname{RL}}
\newcommand{\rsl}{\operatorname{RSL}}
\newcommand{\tw}{\operatorname{tw}}
\newcommand{\nul}{\operatorname{null}}

\newcommand{\mult}{\operatorname{mult}}
\newcommand{\diag}{\operatorname{diag}}
\newcommand{\sgn}{\operatorname{sgn}}

\newcommand{\spec}{\operatorname{spec}}

\newcommand{\tr}{\operatorname{tr}}

\def\B#1{B^{[#1]}}
\def\Ba#1{B_a^{[#1]}}

\newcommand{\lam}{\lambda}

\newcommand{\bit}{\begin{itemize}}
\newcommand{\eit}{\end{itemize}}
\newcommand{\ben}{\begin{enumerate}}
\newcommand{\een}{\end{enumerate}}
\newcommand{\beq}{\begin{equation}}
\newcommand{\eeq}{\end{equation}}
\newcommand{\bea}{\begin{eqnarray*}}
\newcommand{\eea}{\end{eqnarray*}}
\newcommand{\bpf}{\begin{proof}}
\newcommand{\epf}{\end{proof}\ms}
\newcommand{\ms}{\medskip}

\newcommand{\lc}{\left\lceil}
\newcommand{\rc}{\right\rceil}

\newcommand{\du}{\,\dot{\cup}\,}

\title{Rigid linkages and partial zero forcing}

\author{Daniela Ferrero\\
\small Department of Mathematics \\ [-.8ex]
\small Texas State University \\ [-.8ex]
 \small San Marcos, TX 78666, USA \\ [-.8ex]
 \small dferrero@txstate.edu
 \and Mary Flagg \\ 
 \small Department of Mathematics \\ [-.8ex]
 \small  Computer Science and Cooperative Engineering \\ [-.8ex]
\small University of St. Thomas \\ [-.8ex] 
\small 3800 Montrose,   Houston, TX 77006, USA  \\[-.8ex] 
\small flaggm@stthom.edu
\and H.~Tracy Hall \\ 
 \small  NewVistas, LLC \\ [-.8ex] 
 \small Provo UT 84606, USA\\ [-.8ex]
 \small h.tracy@gmail.com
\and Leslie Hogben \\
\small Department of Mathematics \\ [-.8ex]
\small Iowa State University \\ [-.8ex] 
\small Ames, IA 50011, USA  \\ [-.8ex] 
\small hogben@iastate.edu \\   
\small and \\ 
\small American Institute of Mathematics \\ [-.8ex]
  \small 600 E. Brokaw Road \\ [-.8ex]
  \small San Jose, CA 95112, USA \\ [-.8ex] 
   \small hogben@aimath.org
   \and Jephian\,C.-H.\,Lin \\ 
   \small Department of Mathematics \\ [-.8ex] 
   \small   Iowa State University\\[-.8ex] 
     \small  Ames, IA 50011, USA \\[-.8ex] 
      \small  jephianlin@gmail.com
\and Seth\,A.\,Meyer
\\
\small Mathematics Discipline \\[-.8ex] 
 \small St. Norbert College \\[-.8ex] 
\small  De Pere, WI 54115, USA  \\[-.8ex] 
 \small seth.meyer@snc.edu
 \and Shahla Nasserasr\\
\small Department of Mathematics  \\[-.8ex] 
\small Nova Southeastern University  \\[-.8ex] 
\small Ft Lauderdale, FL,  33314, USA \\[-.8ex] 
\small  shahla.nasserasr@gmail.com
\and Bryan Shader \\
\small Department of Mathematics \\ [-.8ex] 
\small  University of Wyoming \\ [-.8ex] 
 \small  Laramie, WY 82071, USA \\ [-.8ex] 
\small  bshader@uwyo.edu}

\begin{document}

\maketitle\vspace{-10pt}

 \begin{abstract} 
Connections between 
  vital linkages 
 and zero forcing
 are established.  Specifically,   the notion of a rigid linkage is introduced as a special kind of unique linkage  and it is shown that spanning forcing paths of a zero forcing process  form a  spanning rigid linkage and thus  a vital linkage. A related generalization of zero forcing that produces a rigid linkage via a coloring process is developed.  One of the motivations for introducing zero forcing is to provide an upper bound on
the maximum multiplicity of an eigenvalue  among the real symmetric matrices described by a graph.
Rigid linkages and a related notion of rigid shortest linkages are utilized to obtain bounds on the multiplicities of eigenvalues  of this family of matrices.
 \\
 {\bf Key words.} Linkage,  Vital, Rigid, Zero Forcing, Inverse Eigenvalue Problem\end{abstract}

\section{Introduction}\label{prelims}

Throughout this paper  $G=(V(G), E(G))$ is a (simple, undirected) graph. We refer the reader to \cite{D17} for standard graph theoretic terminology.   Robertson and Seymour defined unique linkages in \cite{RSxxi}.
   A \emph{linkage} in $G$ is a subgraph whose connected components are paths (as is customary in the literature, a single vertex can be considered as a path, in which case the set of endpoints is a singleton  rather than a pair).  The \emph{order} of a linkage is its number of components.  A linkage $\p$ in a graph $G$ is a \emph{spanning} linkage if $V(\p)=V(G)$. 
 The \emph{pattern} of a linkage $\p=\{p_i\}_{i=1}^t$ is the set $\{\{\alpha_1,\beta_1\},\{\alpha_2,\beta_2\},\ldots,\{\alpha_t,\beta_t\}\}$, where  $\{\alpha_i,\beta_i\}$ is the set of endpoints of $p_i$.  
   Any linkage uniquely determines its pattern.  A linkage is {\em unique} if it is the only linkage with its pattern, and a linkage is \emph{vital} if it is both a unique and a spanning linkage.  

\begin{defn}   
Let $\alpha, \beta\subseteq V(G)$. 
A linkage $\p$ is an \emph{$(\alpha,\beta)$-linkage} if $\alpha$ consists of one endpoint of each path in $\p$ and $\beta$ consists of the other endpoints of the paths.  In the case that the path is a single vertex, the vertex is in both $\alpha$ and $\beta$.
\end{defn}  

Observe that an $(\alpha,\beta)$-linkage is the same as a $(\beta,\alpha)$-linkage.
Note that if $\p$ is an $(\alpha,\beta)$-linkage of order $t$, then necessarily $|\alpha|=|\beta|=t$.   Given a linkage $\p=\{p_i\}_{i=1}^t$, we can construct sets $\alpha$ and $\beta$ such that $\p$ is an $(\alpha,\beta)$-linkage by starting with $\alpha =\beta=\{v_i:p_i=(v_i)\}$ (the set of vertices in one-vertex paths in $\p$).  Then for each $i$ such that $p_i$ has more than one vertex, choose one endpoint of $p_i$ to place in $\alpha$ and place the other endpoint in $\beta$. 
Note that given a linkage $\p$ of order $t$, there can be up to $2^{t-1}$ distinct pairs  $(\alpha,\beta)$ for which $\p$ is an $(\alpha,\beta)$-linkage. 

\begin{defn}   A linkage $\p$ is {\em $(\alpha,\beta)$-rigid} if $\p$ is the unique $(\alpha,\beta)$-linkage in $G$.  A linkage $\p$ is {\em rigid} if $\p$ is $(\alpha,\beta)$-rigid for some $\alpha$ and $\beta$ such that $\p$ is an $(\alpha,\beta)$-linkage.
\end{defn}

\begin{rem}
A unique linkage (or even a vital linkage) may not be  rigid, because for a unique linkage the endpoints of each path are fixed, whereas for an  $(\alpha,\beta)$-rigid linkage only the sets of endpoints are fixed. This is illustrated in Example \ref{exuniquenotrigid}. However, an $(\alpha,\beta)$-rigid linkage is necessarily unique. To see this, suppose  $\p=\{p_i\}_{i=1}^t$ is a rigid linkage with pattern  $\{\{\alpha_i,\beta_i\}:i=1,\dots,t\}$, so for $\alpha=\{\alpha_i\}_{i=1}^t$ and $\beta=\{\beta_i\}_{i=1}^t$, $\p$ is an $(\alpha,\beta)$-rigid linkage. Then any other linkage with the same pattern would contradict that $\p$ is $(\alpha,\beta)$-rigid.  
\end{rem}

\begin{ex}\label{exuniquenotrigid}
Consider the complete graph $K_4$, with the vertices labeled 1, 2, 3, 4. 
The linkage $\p=\{(1,2),(3,4)\}$ is a unique linkage.  There are two ways to make $\p$ an $(\alpha,\beta)$-linkage (up to swapping $\alpha$ and $\beta$).    For $\alpha =\{1,3\}$ and $\beta = \{2,4\}$, $\p$ is not $(\alpha,\beta)$-rigid   because $\p'=\{(1,4),(3,2)\}$ is another $(\alpha,\beta)$-linkage.  For $\hat\alpha=\{1,4\}$ and $\hat\beta=\{2,3\}$,   $\p$ is not $(\hat\alpha,\hat\beta)$-rigid because $\p'=\{(1,3),(4,2)\}$ is another $(\hat\alpha,\hat\beta)$-linkage.  Thus $\p$ is not rigid.
\end{ex}

\begin{obs}  In a tree every linkage is rigid.
\end{obs}

In \cite{RSxxi}, Robertson and Seymour define the set of {\em terminals} of a linkage to be the set of endpoints of the paths in the linkage and define the pattern as a partition of the terminals with two terminals in the same partite set if and only if they are endpoints of the same path. This definition of pattern 
is equivalent to the definition we gave at the beginning, namely  $\{\{\alpha_1,\beta_1\},\{\alpha_2,\beta_2\},\ldots,\{\alpha_t,\beta_t\}\}$, where $\alpha_i$ and $\beta_i$ are the endpoints of the $i$th path in the linkage. The partition of the terminals 
is a set of sets of cardinality one and two.  A set of size one occurs exactly  when a path is a single vertex, in which case  $\alpha_i=\beta_i$ and    $\{\alpha_i,\beta_i\}=\{\alpha_i\}$.   

Zero forcing is a graph coloring process introduced independently in combinatorial matrix theory  \cite{AIM08} and in control of quantum systems \cite{BG07}, and the {\em zero forcing number} $\Z(G)$ is the minimum number of blue vertices needed to color all the vertices of  the graph blue; the formal definition of zero forcing is given later in this section.  Zero forcing also appears as part of the power domination process used to determine optimal placement of  monitoring units in an electric network (see \cite{REUF15}) and in fast-mixed graph searching in computer science (see \cite{Y13}).

A {\it tree decomposition} of a graph $G=(V(G), E(G))$ is a pair $(T, W)$, where $T$ is a tree and 
$W =\{W_x : x \in V(T) \}$ is a collection of subsets of vertices of $G$ 
satisfying 
\begin{itemize}
\item[\rm 1.] Each vertex of $V$ is in at least one $W_x$.
\item[\rm 2.] Each edge of G has both ends in some $W_x$.
\item[\rm 3.] 
Whenever $u,v,w$ are vertices such that $v$ lies on a path in $T$ from $u$ to $w$, then $W_u \cap W_w \subseteq W_v$.
\end{itemize}
The {\it width} of a tree decomposition is one less than the maximum cardinality of a $W_x$.  The {\it tree-width}  of $G$, denoted $\mbox{tw}(G)$, is the minimum width of any tree
decomposition of $G$.
 For a nonnegative integer $p$,  a linkage is a $p$-linkage if the cardinality of its terminal set is less than or equal to $p$ \cite{RSxxi}.  A main result in \cite{RSxxi} is that for every nonnegative integer $p$ there exists an integer $w\ge 0$ such that  every graph with a vital $p$-linkage  has tree-width at most $w$.  We obtain an analogous result that the order of a spanning rigid linkage is an upper bound on tree-width (Corollary \ref{rigidlinktw}).  To establish this, in Section \ref{s:RLforce} we define rigid linkage forcing and show that a linkage is rigid if and only if its paths can be produced by a rigid linkage forcing process (Theorem \ref{Zsharplinkage}).  Furthermore, a spanning rigid linkage is also the set of paths produced by a standard zero forcing process  (Corollary \ref{RLchainZ}), so there is a   rigid linkage of order $\Z(G)$.  

 The  introduction of zero forcing in \cite{AIM08} was motivated by providing an upper bound for the maximum multiplicity of the eigenvalues among the real symmetric matrices having off-diagonal nonzero pattern described by the edges of a given graph $G$.  Maximum eigenvalue multiplicity, although still an open question except for certain families of graphs,  is one part of the more general problem of determining the spectra (multisets of eigenvalues) of this set of matrices described by $G$, which is the  {\em inverse eigenvalue problem} for the graph $G$.  In Section \ref{sevalmult} we use rigid linkage forcing, without forcing the entire graph, to obtain bounds on the multiplicities of eigenvalues of matrices associated with $G$.   In Section \ref{app}, we use the techniques from Section \ref{sevalmult} to give a new infinite family of graphs 
 and multiplicity lists for which the eigenvalues of matrices with the given graph and multiplicity list must satisfy certain linear equations.  
We refine the idea of rigid linkage to rigid shortest linkage to improve these bounds in Section \ref{srsl}, and apply this in Section  \ref{comp}.  
We conclude this section with some basic properties of rigid linkages. 
\begin{obs}
\label{obs:sub}
Let $G$ be a graph and $\p$ be an $(\alpha,\beta)$-rigid linkage in $G$.  Then for any subgraph $H$ of $G$ that contains $\p$ as a subgraph, $\p$ is an $(\alpha,\beta)$-rigid linkage in $H$.
\end{obs}

\begin{prop}\label{rigidsublinkage}
Suppose $\alpha=\{\alpha_i\}_{i=1}^t$, $\beta=\{\beta_i\}_{i=1}^t$, and $\p=\{p_i\}_{i=1}^t$ is an $(\alpha,\beta)$-rigid linkage, with the endpoints of $p_i$ denoted by $\alpha_i$ and $\beta_i$.  Let $\mathcal{I}\subseteq\{1,\ldots ,t\}$.  For $i\in \mathcal{I}$, let $p'_i$ be a subpath of $p_i$ with endpoints $\alpha'_i$ and $\beta'_i$ (with $\alpha_i$, $\alpha'_i$, $\beta'_i$,  $\beta_i$ following the path order of $p_i$).  Define $\alpha'=\{\alpha'_i\}_{i\in\mathcal{I}}$ and $\beta'=\{\beta'_i\}_{i\in\mathcal{I}}$.  Then $\p'=\{p'_i\}_{i\in\mathcal{I}}$ is an $(\alpha',\beta')$-rigid linkage in the graph $G-(V(\p)\setminus V(\p'))$.
\end{prop}
\begin{proof}
Suppose $\p''\neq \p'$ is an $(\alpha',\beta')$-linkage in $G-(V(\p)\setminus V(\p'))$.  Since $\p''$ does not use any vertex in $V(\p)\setminus V(\p')$, we may replace $\p'$ by $\p''$ in $\p$ and obtain another $(\alpha,\beta)$-linkage, which is a contradiction to the fact that $\p$ is $(\alpha,\beta)$-rigid.
\end{proof}


\section{Rigid linkage forcing}\label{s:RLforce}

In this section we define a  new forcing process (related to zero forcing)  that describes a rigid linkage in  terms of forcing paths.
To provide context for the next definition, we first recall the notion of the zero forcing number and its associated color change rule (see \cite{cancun} for further details).
Let $G$ be a graph. Throughout the process each vertex is colored  blue or white,  and one step of the process changes the color of one vertex from 
white to blue by application of the following \emph{color change rule} (CCR):
\begin{itemize}
\item (CCR-$\Z$)
 If $u$ is a blue vertex and exactly one neighbor $w$ of $u$ is white, then  change the color of $w$ to blue. 
\end{itemize}
 The coloring of $w$ by $u$ is called a  \emph{$\Z$-force} and is denoted by $u\to w$.

Given an initial coloring of $G$, the {\it derived set} is a set of blue vertices
obtained by applying the color-change rule until no more changes are possible, and this process is called a {\it zero forcing} process. 
A {\it zero forcing set} for
G is a subset of vertices $B$ such that if initially the vertices in $B$ are colored blue and the remaining
vertices are colored white, then the derived set is all the vertices of $G$. The zero forcing number $\Z(G)$ is
the minimum of $|B|$ over all zero forcing sets $B$ for $G$. 

We now are ready to define a new type of forcing process. 
  For $X\subseteq V(G)$ define the {\em boundary} $\partial_G(X)$ of $X$ to be the set of vertices not in $X$ that have 
at least one neighbor in $X$. 
 When $K$ is a subgraph of $G$, $\partial_G(K)=\partial_G(V(K))$.

\begin{defn}
Let $G$ be a graph. Each vertex is colored blue or white and  one step changes the color of one vertex from white to blue by application of the \emph{color change rule 
CCR-$\RL$} (defined below).   At each step some of the blue vertices are {\em active} and the others are {\em inactive} (an active vertex can become inactive but not vice versa). 
Step $0$ is completed by selecting an initial 
set of blue vertices denoted by $\B 0$ and defining the initial active set of vertices $\Ba 0$ to be $\B0$.
After step $k$ there is a set $\B k$ of blue vertices and an active set of blue vertices $\Ba k$.       An application of {CCR-$\RL$} to go from step $k$ to step $k+1$ (which may involve choice) is:
\ben[(1)]
\item\label{c:CCR-RL-1} Pick a component $K$ of $G-\B k$ such that $\partial_G(K)$ does not contain any inactive blue vertices. 
\item\label{c:CCR-RL-2} Select an  active blue vertex $u$ such that $w$ is the only white neighbor of $u$  in $K$:
\bit 
 \item Color $w$ blue, i.e., $\B {k+1}=\B k\cup\{w\}$.  
\item  $u$ becomes inactive and $w$ becomes active, i.e., $\Ba {k+1}=\Ba k \setminus \{u\}\cup\{w\}$.
\eit
The coloring of $w$ by $u$ is called an  \emph{$\RL$-force} and is denoted by $u\to w$.
\item  Step $k+1$ is completed once (1) and (2) have been performed (i.e., a component $K$ and an active vertex $u$ have 
been chosen, a vertex $w$ is forced, and the set of active vertices is updated).
\een
A {\em rigid linkage forcing process} or  \emph{$\RL$-forcing process} on $G$ is a sequence of steps applying  CCR-$\RL$ to an initial set $\B 0 \subseteq V(G)$ of blue vertices.  
At the end of an $\RL$-forcing process there may still  exist  $\RL$-forces. 
The minimum number of initial blue vertices 
$\B 0$
 such that there is a way to color all vertices blue is called the \emph{$\RL$-forcing number}, and is denoted by $\Zsharp(G)$. 
\end{defn}

As indicated by the name, the goal of $\RL$-forcing is to create a rigid linkage.
The next example shows why rule (1), which  prevents $\RL$-forcing into a white component with an inactive blue neighbor, is necessary.
\begin{figure}[h]\vspace{-10pt}
\begin{center}
\scalebox{1}{
\begin{tikzpicture}
\foreach \i/\ang in {1/45,2/135,3/225,4/315}{
\node[fill=blue, label={[label distance=-5pt]\ang:$\i$}] (\i) at (\ang:1) {};
}
\foreach \i/\ang in {5/0,6/90,7/180,8/270}{
\node[label={[label distance=-5pt]\ang:$\i$}] (\i) at (\ang:1.414) {};
}
\foreach \i/\ang in {9/45,10/135,11/225,12/315}{
\node[label={[label distance=-5pt]\ang:$\i$}] (\i) at (\ang:2) {};
}
\draw (1)--(2)--(3)--(4)--(1);
\draw (5)--(1)--(6)--(2)--(7)--(3)--(8)--(4)--(5);
\draw (5) edge (9)
      (6) edge (10)
      (7) edge (11)
      (8) edge (12);
\end{tikzpicture}}\vspace{-15pt}
\end{center}
\caption{A graph $G$ that illustrates the need to exclude inactive blue vertices for  CCR-RL\label{mary}}
\end{figure}
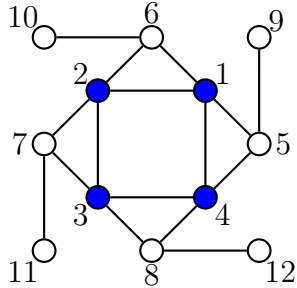
\begin{ex}\label{ex:mary}
Let $G$ be the graph in Figure \ref{mary}.  Suppose we start the $\RL$-forcing process with the initial blue vertices $\Ba 0=\alpha=\{1,2,3,4\}$.  Then we can $\RL$-force $1\rightarrow 6$, $2\rightarrow 7$, $3\rightarrow 8$, $6\to 10$, $7\to 11$, and $8\to 12$,   in that order.  But if the $\RL$-force $1\to 6$ occurs, then $5$ has an inactive blue neighbor, so $5$ cannot be $\RL$-forced. 
 There is no rigid linkage between $\alpha$ and $\beta=\{5,6,7,8\}$, because we have linkages 
\[\{(1,6),(2,7),(3,8),(4,5)\}\text{ and }\{(1,5),(2,6),(3,7),(4,8)\}.\] 
 Note also that any $\RL$-forcing process starting with $\alpha$ cannot force all of $5,6,7,8$ blue;  one of them must be sacrificed, and this is necessary to make the linkage rigid.   
 However, $\Z(G)=4$, because $\{4,9,10,11\}$ is a zero forcing set (and therefore is a rigid linkage forcing set by Proposition \ref{ZisRL} below).
\end{ex}

 In both zero forcing and rigid linkage forcing, there are usually choices to be made as to which possible  force is performed at each step, so an initial set of blue vertices usually results in many different forcing processes.  However, zero forcing has a unique result if  the process is carried out to until no more $\Z$-forces are possible \cite{AIM08}, whereas rigid linkage forcing may produce more than one final set 
 of blue vertices.
 For the graph $G$ in Example \ref{ex:mary},  the final set of blue vertices is $\B 6=\{1,2,3,4,6,7,8,10,11,12\}$ for the $\RL$-forcing process described therein.  If instead we use the $\RL$-forces $1\to 5, 4\to 8, 3\to 7, 5\to 9, 8\to 12, 7\to 11$, then we obtain $\B 6=\{1,2,3,4,5,7,8,9,11,12\}$.  Example \ref{ex:paw} gives  a graph and an initial set  for which there are final sets of blue vertices with different cardinalities. 

\begin{prop}\label{ZisRL}
Any zero forcing process on a graph $G$ is a rigid linkage forcing process.  
\end{prop}
\bpf Initially all blue vertices are active so the first $\Z$-force is a valid $\RL$-force.  Let $\B k$ be the set of blue vertices after step $k$ in a zero forcing process on $G$.  Suppose $w$ is a white vertex, i.e., $w\not\in \B k$. Then every blue neighbor of $w$ is active, because in order to become inactive a vertex must perform an $\RL$-force and no neighbor of $w$ can $\Z$-force a vertex other than $w$ as long as $w$ is white.  Thus whatever $\Z$-force is performed next is a valid $\RL$-force.
\epf
\begin{defn}  Let $G$ be a graph.  For  a given rigid linkage forcing process on $G$ with $r$ steps: 
\bit
\item An {\em $\RL$-forcing chain}  is a path $(v_0,v_1,\dots,v_\ell)$ such that $v_0\in \B 0$, $v_\ell\in \Ba r$,  and $v_i\to v_{i+1}$ for all $i=0,\dots,\ell-1$.
\item  The {\em $\RL$-chain set} is the set of all  $\RL$-forcing chains (for the given forcing process).
\eit
\end{defn} 
Analogous definitions are used for $\Z$-forcing chains and the $\Z$-chain set using $\Z$-forcing. 
It is well known that $\Z$-forcing chains are  disjoint induced paths.
Next we make some basic observations.

\begin{obs}
\label{Zsharpchains}
Suppose an $\RL$-forcing process on $G$ starts with blue vertex set $\B 0$ and stops with blue vertex set $\B r$ and active blue vertex set $\Ba r$\!.  Then each $\RL$-forcing chain is an induced path in $G$ and the chain set is a 
$(\B 0,\Ba r)$-linkage.
\end{obs}

\begin{obs}
\label{obs2.6}
By Proposition $\ref{ZisRL}$, any $\Z$-chain set is an $\RL$-chain set.  
\end{obs}

The next example shows that  making certain choices in an $\RL$-forcing process can result in $\B r\ne V(G)$ even if   $\B 0$ is a zero forcing set and the rigid linkage forcing process is run until no more
$\RL$-forces are possible. 
\begin{ex}\label{ex:paw}  Consider the set $\B 0=\{2,3\}  $ of initial blue vertices in the paw graph shown in Figure \ref{fig:paw}.  Then $\B 0$ is a zero forcing set (with $\Z$-forces $3\to 4$ and $2\to 1$).  For rigid linkage forcing, we have the option to choose $2\to 1$ as the first $\RL$-force, and if we do so, then no more $\RL$-forces are possible.  

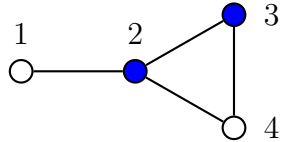
\begin{figure}[h!]
\begin{center}
\begin{tikzpicture}
\node[label={$1$}] (1) at (-1.5,0) {};
\node[fill=blue,label={$2$}] (2) at (0,0) {};
\node[fill=blue,label={right:$3$}] (3) at (30:1.5) {};
\node[label={right:$4$}] (4) at (-30:1.5) {};
\draw (1) -- (2) -- (3) -- (4) -- (2);
\end{tikzpicture}
\caption{The paw graph\label{fig:paw}}
\end{center}\vspace{-15pt}
\end{figure}
\end{ex}

By Observation \ref{Zsharpchains}, any $\RL$-chain set is a linkage.  We will show that being an $\RL$-chain set is equivalent to being a rigid linkage (hence the name `rigid linkage forcing').
The set of neighbors of a vertex $u$ in a graph $G$ is denoted by $N_G(u)$.

\begin{lem}
\label{forcelem}
Suppose $\alpha=\{\alpha_i\}_{i=1}^t$, $\beta=\{\beta_i\}_{i=1}^t$, $\alpha_i\neq \beta_i$ for all $i$, and every component of $G-\alpha$ intersects  $\beta$.  If there exists an $(\alpha,\beta)$-rigid linkage, then there is a vertex $u\in\alpha$ such that $|N_G(u)\setminus\alpha|=1$.
\end{lem}
\begin{proof}
Let $\p=\{p_i\}_{i=1}^t$ be an $(\alpha,\beta)$-linkage.  Since $\alpha_i\neq \beta_i$ for all $i$, every vertex in $\alpha$ has a neighbor not in $\alpha$, namely, the next vertex after $\alpha_i$ on $p_i$.  We  show that if $|N_G(\alpha_i)\setminus\alpha|\geq 2$ for all $\alpha_i\in\alpha$, then $\p$ is not an $(\alpha,\beta)$-rigid linkage.  

Suppose $|N_G(\alpha_i)\setminus\alpha|\geq 2$ for all $\alpha_i\in\alpha$.  Construct a digraph $\Gamma$ on $t$ vertices $\{v_i\}_{i=1}^t$ with arcs $(v_i,v_j)$ whenever there is a path $p_i'\neq p_i$ in $G$ going from $\alpha_i$ to $\beta_j$ and $(V(p_i')\setminus \{\alpha_i\})\cap V(\p)\subseteq V(p_j)$.  Here $i=j$ is possible.  

  Since $|N_G(\alpha_i)\setminus\alpha|\geq 2$, each $\alpha_i\in\alpha$ has a neighbor $u\notin\alpha$ that is not the next vertex after $\alpha_i$ on $p_i$.  If $u\in p_j$, then $(v_i,v_j)$ is an arc in $\Gamma$.  If $u\notin V(\p)$, then there is a path $p'$ from $\alpha_i$ that passes through $u$  and ends at some vertex in $\beta$ without using any vertex in $\alpha$ except for $\alpha_i$, since the component of $G-\alpha$ containing $u$ intersects  $\beta$.  Proceed along the path $p'$ from $u$ until the first  vertex $x$  in $V(\p)$; suppose $x\in p_j$.  Then we have found a path $p_i'$  from $\alpha_i$ to $\beta_j$ with $(V(p_i')\setminus \{\alpha_i\})\cap V(\p)\subseteq V(p_j)$.  This means every vertex of $\Gamma$ has out-degree at least one.

Therefore, $\Gamma$ contains a directed cycle.  Pick a shortest directed cycle $C$ in $\Gamma$.  Say $C$ has vertices $\{v_{i_k}\}_{k=1}^\ell$ and arcs $(v_{i_k},v_{i_{k+1}})$, where the addition is taken modulo $\ell$.  Let $p_{i_k}'$ be the corresponding path connecting $\alpha_{i_{k}}$ to $\beta_{i_{k+1}}$.  If for some distinct $a,b\in\{1,\ldots,\ell\}$ the paths $p_{i_a}'$ and $p_{i_b}'$ intersect, then both $(v_{i_a},v_{i_{b+1}})$ and $(v_{i_b},v_{i_{a+1}})$ are arcs in $\Gamma$, contradicting the fact that $C$ is a shortest directed cycle.  This means the paths $p'_{i_k}$ are vertex-disjoint for $k=1,\dots,\ell$.  Thus the set of paths $\{p'_{i_k}\}_{k=1}^\ell$ together with the paths in $\p\setminus\{p_{i_k}\}_{k=1}^\ell$ forms a linkage different from $\p$.
\end{proof}

\begin{defn}  Suppose  $\alpha,\beta\subseteq V(G)$ with $|\alpha|=|\beta|$.  Consider starting an $\RL$-forcing process with $\Ba 0=\alpha$ with the goal of ending with $\Ba r=\beta$.  After step $k$, when $\B k$ is the set of blue vertices,  we call every vertex in any component of $G-\B k$ that intersects  $\beta$ an {\em effective} vertex. 
\end{defn}
\begin{thm}
\label{Zsharplinkage}
Let $G$ be a graph and $\p$ a linkage in $G$.  Then $\p$ is a rigid linkage if and only if $\p$ is an $\RL$-chain set under some $\RL$-forcing process.
\end{thm}
\begin{proof}

We first prove that any $\RL$-chain set is a rigid linkage by  induction on the number of forces.  Suppose $\p=\{p_i\}_{i=1}^t$ is an $\RL$-chain set,  $\B 0 =\alpha=\{\alpha_i\}_{i=1}^t$, $\Ba r=\beta=\{\beta_i\}_{i=1}^t$, and the endpoints of $p_i$ are $\alpha_i$ and $\beta_i$.  If no forces happen, then $\p$ is composed of paths on one vertex, so it is a rigid linkage.  Otherwise we may assume without loss of generality that   $\beta'_t\rightarrow \beta_t$ is the last force performed.  

Define   $p_t'$ to be the path obtained from $p_t$ by removing   vertex $\beta_t$,  $\p'=\{p_i\}_{i=1}^{t-1}\cup\{p_t'\}$, and $\beta'=\beta\setminus \{\beta_t\}\cup\{\beta_t'\}$.  By the induction hypothesis, $\p'$ is the rigid linkage between $\alpha$ and $\beta'$ in $G$.  Let $K$ be the component of $G-V(\p')$ containing $\beta_t$.  After  step $r-1$, when $\beta'_t\rightarrow \beta_t$ is the next force, $\beta_t$ is the only white neighbor of $\beta'_t$ in $K$ and $\partial_G(K)$ does not contain any inactive vertices since $\beta'_t\to\beta_t$ using CCR-$\RL$.  This means any path $\tilde p_t$ going from $\beta_t$ to some vertex in $\alpha$ must pass through some vertex in $\beta'$.  If such $\tilde p_t$ is in a linkage between $\alpha$ and $\beta$, then $\tilde p_t$ cannot pass through two vertices in $\beta$, so the only choice is that $\tilde p_t$ contains $\beta_t'$ but not any vertex in $\beta\setminus\{\beta_t\}$.  Moreover, $\beta_t$ is the only white neighbor of $\beta_t'$ in $K$, so $\tilde p_t$ must start from $\beta_t$ and  go to $\beta_t'$ immediately.  By the induction hypothesis, the only linkage between $\alpha$ and $\beta\cup\{\beta_t'\}\setminus \{\beta_t\}$ is $\p'$, so $\p$ is the rigid linkage between $\alpha$ and $\beta$.

For the converse, suppose $\alpha=\{\alpha_i\}_{i=1}^t$, $\beta=\{\beta_i\}_{i=1}^t$, and $\p=\{p_i\}_{i=1}^t$ is an $(\alpha,\beta)$-rigid linkage where the endpoints of $p_i$ are $\alpha_i$ and $\beta_i$. 
To show that we can produce an $\RL$-forcing process that has the paths in $\p$ as its $\RL$-forcing chains (starting with $\alpha$ and ending with $\beta$), 
we prove a stronger claim:  After any step of the $\RL$-forcing process, there is an active blue vertex $u$ that can perform an $\RL$-force $u\rightarrow w$ along some path in $\p$, and $w$ is the only white neighbor of $u$ among all effective white vertices, and we use only forces of this type.

After  step $k$ is completed, we  
define $\alpha'$ to be the set of active blue vertices not in $\beta$, and $\beta'$ to be the set of white vertices in $\beta$.  Consider the  subgraph $G'$ induced by $\alpha'$ and the effective white vertices. 
By Proposition \ref{rigidsublinkage}, there is an $(\alpha',\beta')$-rigid linkage in $G'$, 
so Lemma \ref{forcelem} guarantees the existence of a vertex $u\in\alpha'$  and a vertex $w$ such that  $N_{G'}(u)\setminus \alpha' =\{w\}$.
If $p_i$ is the path containing  $u$, then the next vertex after $u$ in path $p_i$ is  in $N_{G'}(u)\setminus \alpha'$, so $w$ must be the next vertex after $u$ on path $p_i$. Since $G'$ contains all effective white vertices, $w$ is the only white neighbor of $u$ among all effective white vertices.  Let $K$ be the component of $G-{\B k}$ that contains $w$.  Then $\partial_G(K)$ does not contain any inactive blue vertex $u'$, for otherwise just before $u'$  made its $\RL$-force $u'\rightarrow v'$, $u'$ had at least two effective white neighbors, namely $v'$ and a vertex in $K$, 
contrary to the procedure in the claim.  Continuing this process, we are able to build a family of $\RL$-forcing chains on $\p$.
\end{proof}

\begin{rem}
\label{ZsharpZ}
Let $\p$ be an $(\alpha,\beta)$-rigid linkage.
The proof of Theorem \ref{Zsharplinkage} actually proves  that $\p$ can be obtained by starting  with  $\B 0=\alpha$ 
 and using $RL$-forces of the type in the  proof of the theorem. That is, 
if  it is possible to perform an $\RL$-force $u\to w$ after step $k$, then $u\to w$ is a valid $\Z$-force within  the subgraph  of $G$ induced by the union of $\B k$ and the set of effective white vertices.
\end{rem}


\begin{cor}\label{RLchainZ}
Let $G$ be a graph and $\p$ a spanning linkage in $G$.  Then the following are equivalent:
\ben[{\rm 1.}]
\item \label{c:rl} $\p$ is a rigid linkage.
\item \label{c:Zrl} $\p$ is an $\RL$-chain set.
\item \label{c:Z} $\p$ is a $\Z$-chain set.
\een
Therefore, $\Zsharp(G)=\Z(G)$ for any graph $G$, and every $\Z$-chain set or   $\RL$-chain set is a vital linkage of the subgraph of $G$ induced by  the set  of blue vertices at the end of the forcing process.
\end{cor}
\begin{proof}
The equivalence of  (\ref{c:rl})  and  (\ref{c:Zrl})  follows from Theorem \ref{Zsharplinkage}.  By Observation \ref{obs2.6}, any $\Z$-chain set is an $\RL$-chain set. 
By Remark \ref{ZsharpZ},  if $\p$ is a rigid linkage then any $\RL$-force $u\rightarrow v$ can be taken to have the property that $u$ is not adjacent to any effective white vertex except for $w$.  Since $V(\p)=V(G)$ ensures that at every step every white vertex is an effective white vertex, $u\rightarrow w$ is a valid $\Z$-force, and $\p$ is a $Z$-chain set. 
\end{proof}

If $G$ has a spanning rigid linkage $\p$ of order $t$, then    $\Z(G)\le t$ because $\p$ is  a $\Z$-chain set.  Since $\tw(G)\le \Z(G)$ \cite{param}, we obtain the next corollary.  

\begin{cor}\label{rigidlinktw}
Let $G$ be a graph and let $\p$ be an order $t$  spanning rigid linkage in $G$.  Then $\tw(G)\le t$.
\end{cor}

$A$ graph $G$ is a \emph{graph  of two parallel paths} \cite{Row12} provided
\begin{itemize}
\item $G$ is not a path, 
\item   there exist two disjoint induced paths of $G$ that together span $G$ and
\item   $G$ can be drawn in the plane with all edges drawn as line segments such that these two paths are parallel  and edges  between the two paths do not cross. 
\end{itemize}

\begin{cor}\label{rigid2}
Let $G$ be a graph that is not a path.  The following are equivalent.

\ben[{\rm 1.}]
\item\label{rigid2:R2} $G$ has a spanning rigid linkage of order $2$.
\item\label{rigid2:Z2} $\Z(G)=2$.
\item\label{rigid2:2PP} $G$ is a graph of two parallel paths.
\een 
\end{cor}

\bpf By \cite{Row12},
 \eqref{rigid2:Z2} $\Leftrightarrow$ \eqref{rigid2:2PP}.  By Corollary \ref{RLchainZ} and the fact that $G$ is not a path, \eqref{rigid2:R2} $\Rightarrow$ \eqref{rigid2:Z2}.  If $G$ is a graph on two parallel paths $p_1$ and $p_2$ drawn horizontally, then $\p=\{p_1,p_2\}$ is an $(\alpha,\beta)$-rigid linkage where $\alpha$ is the left endpoints and $\beta$ is the right endpoints, so \eqref{rigid2:2PP} $\Rightarrow$ \eqref{rigid2:R2}. 
\epf


In \cite{MWZ12}, Mayhew et al.~characterized graphs with a vital linkage of order $2$. We recall some of the terminology they introduced and some of their results:
 A {\em chord} in a linkage $\p$ is an edge that does not belong to any path in $\p$ but it has both endpoints in a single path $p\in \p$. All other edges of $G$ that do not belong to any path of $\p$ are called {\em rung} edges. 
 Given a  linkage $\p$ in $G$, an edge of $G$ that is also in $E(\p)$ is a {\em path edge}.
A
{\it linkage minor} of $G$ for a  chordless linkage $\p$ is a linkage $H$ that is a minor of $G$
obtained by possibly contracting some path edges  of $G$ and  possibly deleting some rung edges of $G$. A graph $G$ has an $XX$ {\em linkage minor}  with respect to an order 2 chordless linkage $\p$  if it
 has a linkage minor isomorphic to  the complete bipartite graph $K_{2,4}$ such that the endpoints of $\p$ are mapped
to the degree-2 vertices of $K_{2,4}$.  Figure \ref{fig:XX} illustrates an $XX$ linkage minor in which the heavier horizontal edges are path edges, and the thinner diagonal edges are rung edges.  In \cite{MWZ12}, 
Mayhew et al.~proved that for any graph $G$, having a vital linkage of order $2$ is equivalent to having a chordless spanning linkage of order $2$ with no $XX$ linkage minor.  

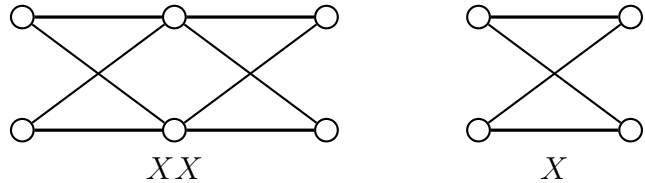
\begin{figure}[h!]
\begin{center}
\begin{tikzpicture}
\begin{scope}
\node (1) at (0,1.5) {};
\node (2) at (2,1.5) {};
\node (3) at (4,1.5) {};
\node (4) at (0,0) {};
\node (5) at (2,0) {};
\node (6) at (4,0) {};
\draw[very thick] (1) -- (2) -- (3);
\draw[very thick] (4) -- (5) -- (6);
\draw (1) -- (5) -- (3);
\draw (4) -- (2) -- (6);
\node[rectangle,draw=none] at (2,-0.5) {$XX$};
\end{scope}
\begin{scope}[xshift=6cm]
\node (r1) at (0,1.5) {};
\node (r2) at (2,1.5) {};
\node (r4) at (0,0) {};
\node (r5) at (2,0) {};
\draw[very thick] (r1) -- (r2);
\draw[very thick] (r4) -- (r5);
\draw (r1) -- (r5);
\draw (r4) -- (r2);
\node[rectangle,draw=none] at (1,-0.5) {$X$};
\end{scope}
\end{tikzpicture}
\caption{An $XX$ linkage minor and an $X$ linkage minor}
\label{fig:XX}
\end{center}\vspace{-15pt}
\end{figure}  
 

   Next we characterize graphs with order 2  spanning rigid linkages in terms of a more restrictive kind of linkage minor.  Let  $G$ be a graph  with a  chordless $(\alpha,\beta)$-linkage $\p=\{p_i\}_{i=1}^t$ such that the endpoints of $p_i$ are $\alpha_i$ and $\beta_i$. An {\em $(\alpha,\beta)$-linkage minor} for $\p$   is a linkage minor for $\p$ in which the points in $\alpha$ and $\beta$ retain their labels.
A graph $G$ has an $X$ {\em $(\alpha,\beta)$-linkage minor} for a chordless linkage $\p$  if there is an $(\alpha,\beta)$-linkage minor for two  paths $p_k$ and $p_\ell$ in  $\p$ that is isomorphic to $C_{4}$ with $\alpha_k$   not adjacent to $\alpha_\ell$  in $C_4$ and $\beta_k$  not adjacent to $\beta_\ell$; see  Figure \ref{fig:XX}.  

\begin{thm}\label{rigid2X}
Let $G$ be a graph with an $(\alpha,\beta)$-linkage $\p=\{p_i\}_{i=1}^t$.  If $\p$ is $(\alpha,\beta)$-rigid, then no path in $\p$ has a chord and  $G$ does not have an $X$  $(\alpha,\beta)$-linkage minor  for $\p$.
Conversely, if 
 $\p$ is an order $2$ spanning chordless  $(\alpha,\beta)$-linkage with no $X$ $(\alpha,\beta)$-linkage minor, then  $\p$ is $(\alpha,\beta)$-rigid.
\end{thm}
\bpf Denote  the endpoints of $p_i\in\p$ by $\alpha_i\in\alpha$ and $\beta_i\in\beta$. 

Suppose first that $p_i$ has a chord $vu$ with $\alpha_i,v,u,\beta_i$ in path order on $p_i$. Then $p_i$ can be replaced by $p'_i$ formed from the subpath $\alpha_i$ to $v$, edge $vu$, and the subpath $u$ to $\beta_i$.  Thus $\p$ is not 
$(\alpha,\beta)$-rigid.

Now suppose $\p$ is chordless and  $G$ has an $X$  $(\alpha,\beta)$-linkage minor  for $\p$.  Without loss of generality assume the $X$ minor involves paths $p_1$ and $p_2$.  
We define  another $(\alpha,\beta)$-linkage $\p'=\{p_i'\}_{i=1}^t$ where $p_i'=p_i$ for $i=3,\dots, t$ and $p'_i$ is defined as follows for $i=1,2$: Define  $f(1)=2$ and $f(2)=1$.
Let $v_i$ be the first vertex (in path order beginning with $\alpha_i$) that has a rung edge forming part of the $X$, and  let $u_{f(i)}$ be the other endpoint of that rung edge; 
thus  $v_i$ precedes $u_i$ in path order from $\alpha_i$ for $i=1,2$.  For $p'_i$ use the  subpath of $p_i$ from $\alpha_i$ to $v_i$, the edge $v_iu_{f(i)}$ and the subpath of $p_{f(i)}$ from $u_{f(i)}$ to $\beta_{f(i)}$. Then $\p$ is not 
$(\alpha,\beta)$-rigid, because $\p'$ is an $(\alpha,\beta)$-linkage. 


 Now assume that $\p$ is an order $2$ spanning chordless   $(\alpha,\beta)$-linkage
and $G$ does not have an  $X$ $(\alpha,\beta)$-linkage minor for $\p$. Then $G$ can be drawn as a graph on two parallel paths with the paths $p_1$ and $p_2$ drawn horizontally.  Then $p_1$ and $p_2$ are $\Z$-chain sets,  where $\alpha$ is the set of left endpoints and $\beta$ is the set of right endpoints.  By Corollary \ref{RLchainZ}, $\p$ is rigid. \epf


\section{Rigid linkages and matrix eigenvalue multiplicities}\label{sevalmult}

In this section we use rigid linkages to establish new bounds on multiplicities of eigenvalues of symmetric matrices described by a graph. 
We begin by recalling needed definitions from the literature.  A graph $G$  with vertex set $\{1,2,\ldots, n\}$ describes the set $\s(G)$ of  real symmetric matrices,  $A=[a_{ij}]$, for which $a_{ij}\neq 0$ when  $ij$ is an edge of $G$,  and $a_{ij}=0$ when $i\neq j$ and $ij$ is not an edge of $G$.  Note that there are no constraints on the diagonal entries of matrices in $\s(G)$.  The multiplicity  of 
the eigenvalue $\lam$ of a real symmetric matrix  $A$ is denoted by $\mult_A(\lam)$; that is, $\mult_A(\lam)= \mbox{dim} \{ {\bf x}\in \mathbb{R}^n: A{\bf x}= \lam {\bf x}\}$.  
The multiset of all eigenvalues of $A$ is denoted by {\it $\spec (A)$}, and we write $\spec(A)= \{ \lambda_1^{(m_1)}, \ldots, \lambda_q^{(m_q)} \} $ 
to indicate that the distinct eigenvalues  of $A$ are $\lambda_1, \ldots, \lambda_q$ with multiplicities $m_1, \ldots, m_q$ respectively.
 The maximum multiplicity of $G$, denoted $\M(G)$,
is the maximum of $\mult_A(\lam)$ over all $A\in \s(G)$ and all eigenvalues of $A$. The minimum rank among matrices in $\s(G)$ is denoted by $\mr(G)$.  As $\lambda I-A \in S(G)$ if and only if $A\in \s(G)$, 
$\mr(G)+\M(G)=n$. 

Given an $n\times n$ matrix,  and subsets $\alpha$ and $\beta$ of $\{1,2,\ldots, n\}$, we use $A[\alpha, \beta]$ (respectively, $A(\alpha, \beta)$)  to 
denote the submatrices of $A$ obtained by keeping (respectively, deleting) the rows with indices in $\alpha$ and columns with indices in $\beta$.  
If $\alpha=\beta$, we simplify the notation to $A[\alpha]$ and $A(\alpha)$.

Let $G$ be a graph and $A=\begin{bmatrix}a_{ij}\end{bmatrix}\in\s(G)$. A {\it linear subgraph} of $G$ is a spanning subgraph whose components are edges or cycles. As a consequence, a graph with isolated vertices does not have a linear subgraph. We introduce the term  generalized linear subgraph to extend the concept of linear subgraph to allow isolated vertices.
A  {\it generalized linear subgraph} of $G$ is a spanning subgraph 
of $G$ each of whose connected components is an isolated vertex, an edge or a cycle.  The main property of linear subgraphs of $G$ is that they correspond to the nonzero terms in the expansion of the determinant of the adjacency matrix of $G$.  Our definition guarantees that generalized linear subgraphs of $G$ maintain this property for all $A \in \s(G)$. 

 Let $C$ be a generalized linear subgraph of $G$. The \emph{weight} of $C$ is 
\[w(C)=\prod_{i\in C_I} a_{ii}\prod_{\{i,j\}\in C_E} a_{ij}^2\prod_{\{i,j\}\in E(C)\setminus C_E}a_{ij} ,\]
where $C_I$ is the set of isolated vertices of $C$, and $C_E$ is the set of  components of order two.  

Next, we  introduce additional notation. The  symmetric group of   $\{1,\ldots, t\}$ is denoted by $\mathfrak{S}_t$.
The {\it disjoint union} $G \du H$ of the graphs $G$ and $H$ is the graph formed by making the vertex sets of $G$ and $H$ disjoint: $V(G \du H)= V(G) \du V(H$) and $E(G\du H)= E(G) \du E(H)$. 

\begin{defn}
Let $G$ be a graph, $A=\begin{bmatrix}a_{ij}\end{bmatrix}\in\s(G)$, and let $\alpha=\{\alpha_i\}_{i=1}^t$ and $\beta=\{\beta_i\}_{i=1}^t$ be  subsets of $V(G)$.  Let $\sigma$  be a permutation of $\{1,\ldots,t\}$.  An \emph{$(\alpha,\beta_\sigma)$-linear subgraph} of $G$ is a spanning subgraph $H=\p\du C$ of $G$ such that $\p$ is an $(\alpha,\beta)$-linkage with pattern $\{\{\alpha_i,\beta_{\sigma(i)}\}: i=1,\dots,t\}$, and $C$ is a generalized linear subgraph of $G -V(\p)$.
The number of cycles of $H$ is denoted by $c(H)$.  Define 
\[w(\p)=\prod_{\{i,j\}\in E(\p)} a_{ij} \text{ and }w(H)=w(\p)w(C).\]

\end{defn}

Observe that if $G$ does not have isolated vertices, it is still possible for $G-V(\p)$ to have isolated vertices. Moreover, since a path in a linkage could have exactly one vertex, an isolated vertex in $G$ is not necessarily in $C$. Notice that isolated vertices in $\p$ do not contribute to any weight to $w(\p)$, but isolated vertices in $C$ do contribute to $w(C)$.

The following theorem gives a combinatorial formula for a given minor of a matrix $A \in \s (G)$.  It can be proven using arguments similar to those in 
\cite{BLP} which gives the formula in the special case that $A$ is the adjacency matrix of $G$.  We note that a combinatorial description of the determinant of the 
adjacency matrix was given earlier in \cite{H}; we find the language in \cite{BLP} better fits the setting of rigid linkages.

\begin{thm}\label{cycledet}
Let $G$ be a graph on $\{1,\ldots,n\}$, $A\in\s(G)$, and $\alpha=\{\alpha_1, \ldots, \alpha_t\}$ and  $\beta=\{\beta_1, \ldots, \beta_t\}$ be $t$-element  subsets of $V(G)$.  Then 
\[\det(A(\alpha,\beta))=(-1)^{\sum_{i=1}^t(\alpha_i+\beta_i)}\sum_{\sigma\in\mathfrak{S}_t}\sgn(\sigma)\sum_{H\in\mathcal{L}_\sigma}(-1)^{|E(H)|}(-2)^{c(H)}w(H), \]
where  $\mathcal{L}_\sigma$ is the set of $(\alpha,\beta_{\sigma})$-linear subgraphs of $G$.  
\end{thm}
\noindent
We note that if $\alpha=\emptyset$ and $\beta=\emptyset$, 
then Theorem \ref{cycledet} gives the formula:
\[\det(A) =\sum_{H}(-1)^{|E(H)|}(-2)^{c(H)}w(H), \]
where the sum is over all generalized linear subgraphs $H$ of $G$.

\begin{lem}\label{determinant}
Let $G$ be a graph having  an $(\alpha,\beta)$-rigid linkage $\p$ and $A\in \s (G)$. Then 
\[\det(A(\alpha,\beta))=\pm w(\p) \det(A(V(\p))). \] 
\end{lem}
\begin{proof}
Since $\p$ is a rigid $(\alpha,\beta)$-linkage in $G$,  the identity map is the 
only permutation $\sigma$ for which there is an $(\alpha,\beta_{\sigma})$-linear subgraph of $G$.
Thus, by Theorem \ref{cycledet}, 
\bea
\det A(\alpha,\beta) 
&=&
(-1)^{\sum_{i=1}^t(\alpha_i+\beta_i)}\sum_{H \in \mathcal{L}_{\tiny \mbox{id}}}
(-1)^{|E(H)|}(-2)^{c(H)}w(H) \\
 &=& \pm (-1)^{\sum_{i=1}^t(\alpha_i+\beta_i)}
 \sum_{H=\p \du C \in \mathcal{L}_{\tiny \mbox{id}}}(-1)^{|E(C)|}(-2)^{c(C)}w(\p)w(C) \\
 &=& 
 \pm w(\p)
 \sum_{ C } (-1)^{|E(C)|}(-2)^{c(C)}w(C) \\
 &=& \pm w(\p)\det(A(V(\p))), 
\eea
where the final sum is over all generalized linear subgraphs of $G-V(\p)$.
\end{proof}

\begin{thm}\label{nullity}
Let $\p$ be an $(\alpha,\beta)$-rigid linkage of order $t$ in a graph $G$, and let  $A\in\mathcal{S}(G)$. 
Then
\[\nul(A(V(\p)) \geq \nul(A)-t.\]
\end{thm}
\begin{proof}
Let $B=A(V(\p))$. A classical fact (see e.g.~Corollary 8.9.2 in \cite{GR01}) is that  each symmetric matrix of rank $r$ has an invertible principal submatrix whose order 
is $r$.  Thus,  there exists a set of indices $\gamma$ such that
$B(\gamma)$ is invertible and $|\gamma|= \nul(B)=\nul A(V(\p))$.
 
By Observation \ref{obs:sub}, $\p$ is an $(\alpha,\beta)$-rigid linkage of $G-\gamma$.
  By Lemma \ref{determinant} applied to $A(\gamma)$, 
\[\det(A(\alpha\cup\gamma,\beta\cup\gamma))=\pm w(\p) \det(A(V(\p)\cup\gamma))=\pm w(\p) \det(B(\gamma))\neq 0.\]
Thus $A(\alpha,\beta)$ has an invertible submatrix of order $n-t-|\gamma|$, and 
\[
 \nul(A(V(\p)))= |\gamma| \geq \nul(A(\alpha,\beta)).\]  Since deleting a row and a column at the same time can change the nullity by at most one, 
\[ \nul (A(\alpha,\beta)) \geq \nul(A)-t, \]
and the result follows.
\end{proof}

\begin{cor}
\label{multiplicity}
Let $\p$ be an $(\alpha,\beta)$-rigid linkage of order $t$ in a graph $G$. Then, for any $A\in\mathcal{S}(G)$ and 
eigenvalue $\lam$ of $A$  
\[\mult_{A(V(\p))}(\lam) \geq \mult_{A}(\lam)-t.\]
\end{cor}

\bpf
Let $A \in \s (G)$ and let $\lam$ be an eigenvalue of $A$.
 By Theorem \ref{nullity} applied to the matrix $B=A-\lambda I$, we 
  conclude \[\nul (B(V(\p))) \geq \nul(B)-t.\] 
 and hence
\[\mult_{A(V(P))}{(\lam})= \nul (B(V(\p))) \geq \nul (B)-t=\mult_{A}({\lam})-t.  \qedhere \]

\epf

Let $G$ be a graph and $A\in \mathcal{S}(G)$.  We use  $m_i(A)$ to denote the $i$th largest multiplicity of an eigenvalue of $A$ for $i=1,\dots,q$, where $q=q(A)$ is the number of distinct eigenvalues of $A$.  Observe that $m_1(A)$ is the maximum multiplicity of an eigenvalue of $A$, denoted by $\M(A)$.  Similarly, let $q_i(A)$ denote the number of distinct eigenvalues of multiplicity at least $i$ for $i=1,\dots,\M(A)$.  Thus $q_1(A)=q(A)$, and the partitions of $|V(G)|$ given by $q_i(A)$ and $m_i(A)$ are conjugate. 
Define $q(G)=\min\{q(A):A\in\s(G)\}$.

\begin{defn}
The {\em rigid linkage number}, denoted by $\rl_G(t)$ (or $\rl(t)$ if $G$ is clear), is  the maximum number of vertices in an order $t$ rigid linkage of $G$.
\end{defn}

\begin{thm}\label{thm:sharpG}
Let $G$ be a graph, $A\in\s(G)$, $q_j=q_j(A)$, and $t$ be a positive integer.  Then
\begin{itemize}
\item[\rm (i)]
 $\sum_{j=1}^t q_j\geq \rl(t)$ and  
 \item[\rm (ii)]  $q(G)\geq \lc \frac{\rl(t)}{t}\rc$.
\end{itemize}
\end{thm}
\begin{proof}
Let $\mu_1, \ldots, \mu_{q_1}$ be the distinct eigenvalues of $A$ with $\mult_{A}(\mu_j)=m_j(A)$.
Choose a rigid linkage $\p$ of order $t$ such that $|V(\p)|=\rl(t)$.   
By Corollary \ref{multiplicity},
\[
\mult_{A(V(\p))}({\mu_j}) \geq m_j(A)-t.
\]
 Since the sum of the multiplicities of all eigenvalues of $A(V(\p))$ is $n-\rl(t)$,
\begin{equation}
\label{one}
 n-\rl(t)\ge  \sum_{j=1}^{q_1} \max\{m_j(A)-t,0\}.
 \end{equation}
It is easy to see from the conjugacy of the partitions that 
\begin{equation}
\label{two} \sum_{j=1}^{q_1} \max\{m_j(A)-t,0\}=q_{t+1}+\cdots+q_{\M(A)},
\end{equation}
 and since $\sum_{j=1}^{\M(A)} q_j=n$, inequality (i) follows from  (\ref{one}) and (\ref{two}).

Inequality (ii) follows from  (i) by letting $A$ range over $\s (G)$ and 
 by noting that for each $A$ the sequence of $q_j$ is non-increasing, so $q_1t\geq q_1+q_2+\cdots +q_t\geq \rl(t)$.
\end{proof}

In the next example we show that the inequality in Theorem \ref{thm:sharpG} (i) is tight.   We exploit the fact that in a tree every linkage is rigid.
\begin{ex}\label{ex:whirl}

Consider the tree   $W$ shown in Figure \ref{fig:whirl}, introduced in \cite{BF} (where it is called $T_2$) and described as a generalized 3-whirl  in \cite{KS}. The center vertex of $W$ is $v_0$ and the 
 paths  $(i_{k2},i_{k1},v_k)$ ($k=1,2,3)$,  $(j_{k2},j_{k1},v_k)$ $(k=1,2)$,   and $(j_{31},v_3)$ are called {\it legs}.  
The rigid linkage numbers $\rl(t)$ are as follows.
	\begin{itemize}
		\item For $t=1$,  $\p_1=\lbrace(i_{12}, i_{11}, v_1, v_0, v_3, i_{31}, i_{32})\rbrace$ is a rigid linkage and is the longest path in $W$. Thus, $\rl(1)=7$.
		\item For $t=2$, $\p_2= \p_1 \cup \lbrace(i_{22}, i_{21}, v_2, j_{21}, j_{22})\rbrace$ is a rigid linkage with $|\p_2|=12$, and this is the largest linkage with two paths. To see this, note that only one path may pass through the center vertex, and if one path contains $v_0$, the longest second path connects two adjacent legs. If neither path in an order 2 linkage contains $v_0$, then the largest number of vertices it can contain  is 10.  Thus, $\rl(2)=12$. 
		\item For $t=3$, the linkage $\p_3= \p_2 \cup \lbrace(j_{12},j_{11})\rbrace$ is a rigid linkage with 14 vertices. No rigid linkage with three paths  will span $W$, since removing a path that includes the center vertex splits the remaining graph  into at least  3 components, which are impossible to cover with 2 paths. Thus, $\rl(3)=14$. 
		\item For $t=4$, the rigid linkage $\p_4=\p_3 \cup \lbrace j_{31} \rbrace$ is a linkage of order 4 that spans the graph. Thus, $\rl(4)=15=|V(W)|$. 
	\end{itemize}
In \cite{BF}, Barioli and Fallat  exhibit a matrix in $\s (W)$ with $q_1=7$, $q_2=5$, $q_3=2$ and $q_4=1$. This matrix results in equality in (i) in Theorem \ref{thm:sharpG} for every $t=1,\dots,4$:
\bea
q_1 =  \phantom{..}7=\rl(1)\phantom{.}\\
	q_1+q_2=12=\rl(2)\phantom{.}\\
	q_1+q_2+q_3=14=\rl(3)\phantom{.} \\
	q_1+q_2+q_3+q_4=15=\rl(4)\smash{.}
\eea
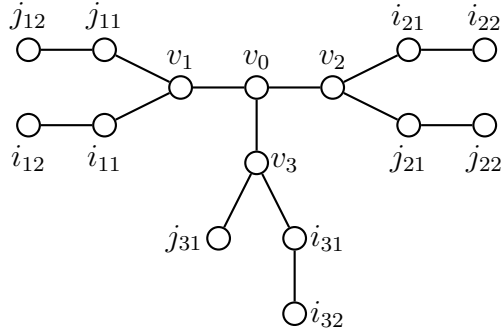
\begin{figure}[h]\vspace{-10pt}
	\begin{center}
		\scalebox{1}{
			\begin{tikzpicture}
			\newcommand*{\spacing}{1}
			\node[label={[label distance=-5pt]$v_0$}] (v0) at (0,0) {};
			\node[label={[label distance=-5pt]$v_2$}] (v2) at (\spacing,0) {};
			\node[label={[label distance=-5pt]$v_1$}] (v1) at (-\spacing,0) {};
			\node[label={[label distance=-5pt]0:$v_3$}] (v3) at (0,-\spacing) {};
			\foreach \k/\dist in {1/1,2/2}{
				\pgfmathsetmacro{\offset}{\dist+\spacing}
				\node[label={[label distance=-5pt]$j_{1\k}$}] (j1\k) at (-\offset,.5) {};
				\node[label={[label distance=-5pt]270:$i_{1\k}$}] (i1\k) at (-\offset,-.5) {};
				\node[label={[label distance=-5pt]$i_{2\k}$}] (i2\k) at (\offset,.5) {};
				\node[label={[label distance=-5pt]270:$j_{2\k}$}] (j2\k) at (\offset,-.5) {};
				\node[label={[label distance=-5pt]0:$i_{3\k}$}] (i3\k) at (.5,-\offset) {};			}
			\pgfmathsetmacro{\secondoffset}{1+\spacing}
			\node[label={[label distance=-5pt]180:$j_{31}$}] (j31) at (-.5,-\secondoffset) {};
			
			\draw (v1)--(v0)--(v2);
			\draw (v0)--(v3);
			\draw (v1)--(j11)--(j12);
			\draw (v1)--(i11)--(i12);
			\draw (v2)--(j21)--(j22);
			\draw (v2)--(i21)--(i22);
			\draw (v3)--(i31)--(i32);
			\draw (v3)--(j31);
			\end{tikzpicture}}\vspace{-15pt}
	\end{center}
	\caption{The tree $W$ that illustrates simultaneous equalities  in (i) of Theorem \ref{thm:sharpG}.\label{fig:whirl}}
\end{figure}
\end{ex}

\begin{cor}
\label{cor:tightRL}
Let $G$ be a graph and $A\in\s(G)$.  Let $q_i=q_i(A)$.  If $\rl_G(t)=\sum_{i=1}^tq_i$ and $\p$ is a rigid linkage of order $t$ with $|V(\p)|=\rl_G(t)$, then 
\[\spec(A(V(\p))=\{\lambda^{(m_\lambda-t)}: \lambda\in\spec(A)\text{ with }m_\lambda>t\},\]
where $m_\lambda$ is shorthand for  $\mult_A(\lam)$. \end{cor}
\begin{proof}
From Corollary \ref{multiplicity}, $\lambda$ is an eigenvalue of $A(V(\p))$ with multiplicity at least $m_\lambda-t$.
 Thus, 
\[\spec(A(V(\p)))\supseteq\{\lambda^{(m_\lambda-t)}: \lambda\in\spec(A)\text{ with }m_\lambda>t\}.\]
The order of $A(V(\p))$ is 
\[|V(G)|-\rl_G(t)=|V(G)|-\sum_{i=1}^t q_i=\sum_{i=t+1}^{\M(G)}q_i=\sum_{\substack{\lambda\in\spec(A)\\m_\lambda>t}}(m_\lambda-t).\]
Therefore  $\spec(A)$ is exactly the described spectrum.
\end{proof}

\section{An application of rigid linkages}\label{app}

The Barioli-Fallat tree $G$ is a graph obtained by appending two leaves to each existing leaf of  $K_{1,3}$.  It is known \cite{Hog05} that if $A\in\s(G)$ with 
$\spec(A)=\{\lambda_{1},\lambda_{2}^{(2)},\lambda_3^{(4)},\lambda_4^{(2)},\lambda_5\}$,
then $\lambda_{2}+\lambda_4=\lambda_{1}+\lambda_5$.  Restrictions of this type were first discovered by Barioli and Fallat \cite{BF}. 
In this section, we give an  infinite family of trees $T_k$  ($k\geq 3$) 
and use rigid linkages to  show that 
the eigenvalues of each $A\in \s(T_k)$ achieving a specific multiplicity list satisfy a linear relationship. 

Let $k$ be a positive integer. The graph $T_k$ shown in Figure \ref{fig:T_k} is the tree obtained from $K_{1,3}$  and $3k$ copies of $K_{1,3}$  by identifying each pendant vertex  of the initial $K_{1,3}$
 with a pendant vertex of $k$ of the copies of $K_{1,3}$.  
Thus $T_k$ has $9k+4$ vertices.  For convenience, we label the central vertex by $0$, 
and its neighbors by $1$, $2$, $3$.  For $\ell=0,1,2,3$ we say that a vertex is in \emph{level $\ell$} if the distance between the vertex and $0$ is $\ell$. Thus $0$ is the only vertex at level $0$; 1,2,3 are the vertices at level $1$; there are $3k$ vertices at level $2$; and 
the $6k$ pendant vertices are the vertices at level $3$.

We first show that  for $k\geq 2$ there is a matrix in $S(T_k)$  whose multiplicity list is
$3k+2, 3k-2, 3k-3, 2,2,1,1,1$.  Throughout this section  the graph $H_k$ shown in Figure \ref{fig:T_k}   is graph obtained by identifying one pendant vertex for each of  $k$ copies of $K_{1,3}$
at a common vertex. Thus, each branch of $T_k$ at $0$ is an $H_k$.

\begin{figure}[!htbp] 
\begin{center}
\begin{tikzpicture}
\node (0) at (0,0) {};
\foreach \i in {0,1,2}{
\pgfmathsetmacro{\ang}{120*\i}
\begin{scope}[rotate=\ang,yshift=-0.6cm]
\node (\i-b) at (0,0) {};
\node (\i-l) at (-1.5,-1) {};
\node (\i-r) at (1.5,-1) {};
\node (\i-ll) at (-1.8,-1.5) {};
\node (\i-lr) at (-1.2,-1.5) {};
\node (\i-rl) at (1.2,-1.5) {};
\node (\i-rr) at (1.8,-1.5) {};
\draw (\i-l) -- (\i-b) -- (\i-r);
\draw (\i-ll) -- (\i-l) -- (\i-lr);
\draw (\i-rl) -- (\i-r) -- (\i-rr);
\draw (\i-b) -- (-0.4,-0.8);
\draw (\i-b) -- (0.4,-0.8);
\draw[dotted] ([xshift=0.5cm]\i-l.center) --node[midway,rectangle,draw=none,fill=white]{$k$} ([xshift=-0.5cm]\i-r.center);
\draw (0) -- (\i-b);
\end{scope}
}
\node[rectangle,draw=none] at (0,-2.6) {$T_k$};
\end{tikzpicture}
\hfil
\begin{tikzpicture}
\node (b) at (0,0) {};
\node (l) at (-1.5,-1) {};
\node (r) at (1.5,-1) {};
\node (ll) at (-1.8,-1.5) {};
\node (lr) at (-1.2,-1.5) {};
\node (rl) at (1.2,-1.5) {};
\node (rr) at (1.8,-1.5) {};
\draw (l) -- (b) -- (r);
\draw (ll) -- (l) -- (lr);
\draw (rl) -- (r) -- (rr);
\draw (b) -- (-0.4,-0.8);
\draw (b) -- (0.4,-0.8);
\draw[dotted] ([xshift=0.5cm]l.center) --node[midway,rectangle,draw=none,fill=white]{$k$} ([xshift=-0.5cm]r.center);
\node[rectangle,draw=none] at (0,-2) {$H_k$};
\end{tikzpicture}
\end{center}
\caption{The graph $T_k$ and the graph $H_k$}
\label{fig:T_k}
\end{figure}
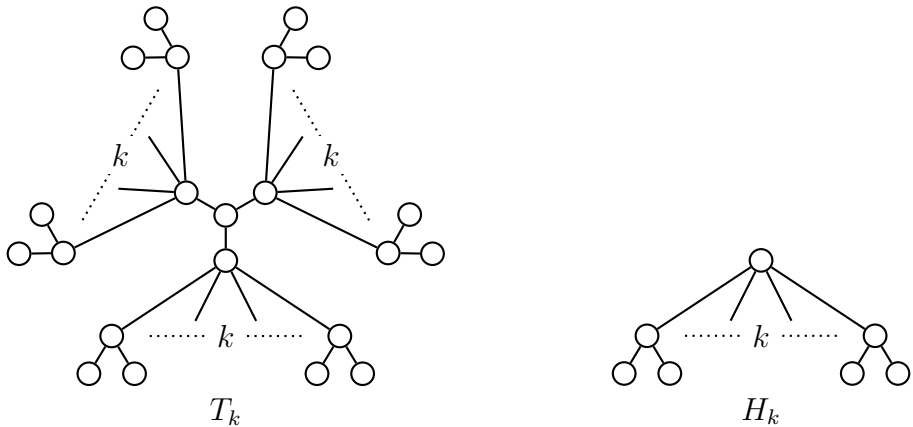

\begin{prop}
\label{specH}
Let $A$ be the adjacency matrix of $H_k$.  Then 
\[\spec(A)=\{0^{(k+1)}, \pm \sqrt{2}^{(k-1)}, \pm \sqrt{k+2} \}.\]
\end{prop}

\bpf
The principal submatrix of $A$ obtained by deleting each vertex adjacent to the central vertex in $H_k$ is 
$O_{2k+1}$.  By eigenvalue interlacing, $0$ is an eigenvalue of  $A$ having multiplicity at least $2k+1-k=k+1$.

The principal submatrix of $A$ obtained by deleting the central vertex is a direct sum of 
$k$ copies of  the adjacency matrix of $P_3$.  Hence $\pm \sqrt{2}$  are eigenvalues of this submatrix with multiplicity $k$.  By eigenvalue 
interlacing,  both $\pm \sqrt{2}$
are eigenvalues of $A$ of multiplicity at least $k-1$.  

The above accounts for all but two of  eigenvalues of $A$.  Since $H_k$ is bipartite, these two eigenvalues
are opposites.  Using the fact that the sum of the squares of the eigenvalues of an adjacency matrix of 
a graph is equal to twice the number of edges, it readily follows that the spectrum is as claimed. 
\epf

\begin{prop}

Let $B$ be the adjacency matrix of $T_k$ with $k\geq 2$, and let $E$ be the matrix obtained from $B$
by replacing its $(1,1)$-entry by $\sqrt{2}$. Then 
$E$  has multiplicity list $3k+2, 3k-2, 3k-3, 2,2,1,1,1$.
\end{prop}

\bpf
 First let $\alpha$ be the set of vertices at level 0 or level 2.
Then $|\alpha |= 3k+1$, and $E(\alpha)=O_{6k+3}$.
It follows from eigenvalue interlacing that $\mult_E(0) \geq 6k+3-(3k+1)=3k+2$. 
As the path cover number of $T_k$ is $3k+2$, $\M(T_k) = 3k+2$, and hence $\mult_E(0)= 3k+2$.

Next let $\beta$ be the set of vertices at level $1$. 
Then $\sqrt{2}$ (respectively, $-\sqrt{2}$) is an  eigenvalue of $E(\beta)$  with multiplicity $3k+1$
(respectively, $3k$). It follows  from eigenvalue interlacing that 
$\mult_E( \sqrt{2}) \geq 3k+1-3=3k-2$, and $\mult_E(-\sqrt{2}) \geq 3k-3$.
 By Proposition  \ref{specH}, the  multiplicity of $\pm \sqrt{2}$ as an eigenvalue of $E(0)$ is $3k-3$.  Thus $\mult_E(\sqrt{2})=3k-2$.
Deleting the vertices of the linkage consisting of  three disjoint maximal paths of length $4$
gives a  principal submatrix of $E$ having $-\sqrt{2}$ as an eigenvalue of multiplicity $3(k-2)$.  
Hence, by Corollary \ref{multiplicity}, $\mult_E(-\sqrt{2})\leq 3(k-2)+3=3k-3$.  Therefore, we conclude
that $\mult_E(-\sqrt{2})=3k-3$.

Note that $\pm \sqrt{k+2}$ are eigenvalues of multiplicity 3 of $E(0)$, and are not eigenvalues of 
the submatrix obtained by deleting the vertices of the linkage consisting of the   disjoint union of a  maximal path of
length 6 and a maximal path of length 4. Hence, by Corollary \ref{multiplicity},  $\mult_E(\pm \sqrt{k+2})=2$.

Each eigenvalue of the  principal submatrix of $E$ obtained by deleting
the vertices of  
a maximal path of length $6$ belongs to the set $\{0, \pm \sqrt{2}, \pm \sqrt{k+2} \}$.
Hence, by Corollary \ref{multiplicity}  each of the remaining eigenvalues of $E$ is simple.  The result
now follows by noting that the sum of the multiplicities for $E$ equals $9k+4
=3k+2+3k-2+3k-3 +2+2+1+1+1$.
\epf
\begin{thm}
Let $k \geq 3$, and suppose that $B \in \s(T_k)$ has spectrum 
\[ 
\{ \lambda_1^{(3k+2)}, \lambda_2^{(3k-2)}, \lambda_3^{(3k-3)}, 
\lambda_4^{(2)}, \lambda_5^{(2)}, 
\lambda_6^{(1)}, \lambda_7^{(1)}, \lambda_9^{(1)} \}.
\]
Then 
\[
\lambda_1+3\lambda_2+3\lambda_3= 2\lambda_4+2\lambda_5+\lambda_6 
+ \lambda_7 + \lambda_8.\]\end{thm}

\bpf
Throughout $P_1$, $P_3$, $P_5$ and $P_7$ will denote paths of $T_k$
between two pendant vertices of $T_k$ that are distance 0, 2, 4 or 6 apart respectively. We make a sequence of observations and repeatedly use Corollary \ref{multiplicity} to  constrain the entries of $B$.

Deleting a linkage $\p$ formed by a $P_7$  and $P_5$ results
in the disjoint union of $(3k-4)$ $P_3$'s and 4 $P_1$'s. 
As each eigenvalue of a matrix whose graph is a path is simple, $\mult_{B(V(\p))}(\lam_1) \leq 3k$.  Because $\mult_B(\lam_1)=3k+2$, Corollary \ref{multiplicity} implies that $\mult_{B(V(\p))}(\lam_1) \geq 3k+2-2= 3k$. Thus each of the  $3k$  principal
submatrices of $B$ corresponding to the components of 
$T_k \setminus (P_7 \du P_5)$ has $\lam_1$ as an eigenvalue. 
By varying the choice of the $P_7$ and $P_5$, we conclude that $b_{vv}= \lambda_1$ whenever $v$ is a pendant vertex of $T_k$.

Deleting a linkage  consisting of $3k$ $P_3$'s and a path joining
two vertices at level 1 results in a graph that is a single vertex in level $1$.
By varying the choice of the two vertices at level 1, it follows from Corollary \ref{multiplicity}  that $b_{11}=\lam_1$, $b_{22}=\lam_1$, and $b_{33}=\lam_1$. 

Deleting a linkage  consisting of three disjoint $P_5$'s results 
in a disjoint union of $(3k-6)$ $P_3$'s, 6 isolated vertices at level 3, 
and one isolated vertex at level 0.   We have shown that the submatrices of $B$ 
corresponding to the isolated vertices each have eigenvalue $\lambda_1$.
As $\mult_B(\lam_2)= 3k-2$,  Corollary \ref{multiplicity} implies that $\lam_2$ is 
an eigenvalue of each $B[V(P_3)]$, and $b_{00}=\lam_2$.  Thus there are at most $3k-6$ components that can contain $\lam_3$ as a (simple) eigenvalue. 
By varying the choice of the  $P_5$'s,  Corollary \ref{multiplicity} implies that $\lam_3$ is 
an eigenvalue of each $B[V(P_3)]$.

Then each principal submatrix of $B$ corresponding to a $P_3$  has trace $\lam_1+\lam_2 + \lam_3$, each principal submatrix of $B$ corresponding to a single vertex at level
1 has trace $\lam_1$, and the principal submatrix of $B$ corresponding to vertex $0$
has trace $\lam_2$.
Hence, on the one hand we have 
\[
 \mbox{tr}(B)= 3k(\lam_1+ \lam_2+\lam_3) + 3\lam_1 + \lam_2.
\]
On the other hand since we know the spectrum of $B$ we have 
\[ \mbox{tr}(B)= (3k+2) \lam_1 + (3k-2) \lam_2 + (3k-3)\lam_3 + 2\lam_4+ 2\lam_5 + \lam_6+\lam_7 + \lam_8.
\]
Therefore
\[
\lam_1+3\lam_2+3\lam_3= 2\lam_4+ 2 \lam_5 + \lam_6+\lam_7+\lam_8.\qedhere 
\]
\epf
\section{Rigid shortest linkages and matrix eigenvalue multiplicities} \label{srsl}

In this section we introduce a new kind of rigid linkage and use it to improve the bounds   in Theorem \ref{thm:sharpG}  for the multiplicities of eigenvalues of symmetric matrices described by a graph. 

\begin{defn}
Let $\alpha,\beta$ be sets of vertices in a graph $G$ such that an $(\alpha,\beta)$-linkage exists, and let $s$ be the smallest number of vertices $V(\p)$ among all $(\alpha,\beta)$-linkages $\p$.  An $(\alpha,\beta)$-linkage is called an \emph{$(\alpha,\beta)$-rigid shortest linkage} if it is the unique $(\alpha,\beta)$-linkage on $s$ vertices.  A linkage $\p$ is called a rigid shortest linkage if it is an $(\alpha,\beta)$-rigid shortest linkage for some $(\alpha,\beta)$.
\end{defn}

\begin{defn}
The {\em rigid shortest linkage number}, denoted by $\rsl_G(t)$ (or $\rsl(t)$ if $G$ is clear), is  the maximum number of vertices in an order $t$ rigid shortest linkage of $G$.
\end{defn}
 
\begin{rem}\label{rem:RSL-RL} A rigid linkage is a rigid shortest linkage so $\rsl(t)\ge \rl(t)$. The reverse inequality is false: A cycle of order $n\ge 3$ is an example, since $\rsl(1)=\lc \frac {n} 2\rc$ and $\rl(1)=1$.  However, a rigid shortest linkage that spans the graph is a rigid linkage, since all vertices are used by the linkage.   By Corollary \ref{RLchainZ}, any spanning rigid linkage is a spanning $\Z$-chain set, and thus has order at least $\Z(G)$.  Thus $\rsl(t)=|V(G)|$ implies $t\ge \Z(G)$.
\end{rem}

\begin{thm}\label{rslthm}
Let $G$ be a graph. 
Then 
\[ {\rm (i)}\, \sum_{i=1}^t q_i\geq \rsl(t)\ \ \mbox{ and }\ \ { \rm (ii)} \ q(G)\geq \lc\frac{\rsl(t)}{t}\rc,\]
for every $A\in\s(G)$ with  $q_i=q_i(A)$  
and every positive integer $t$.
\end{thm}
\begin{proof}
 Let $\alpha=\{\alpha_i\}_{i=1}^t$, $\beta=\{\beta_i\}_{i=1}^t$ be subsets of $V(G)$ such that $\p$ is an $(\alpha,\beta)$-rigid shortest linkage with $|V(\p)|=s$.  
Let $A\in \s(G)$ and set $A_x=xI-A$ viewed as a matrix over the ring of polynomials.  By Theorem \ref{cycledet}, a summand in the formula of $\det(A_x(\alpha,\beta))$ is a polynomial of degree $j$ if the corresponding linear subgraph $H$ has $j$ isolated vertices in the  generalized cycle part of $H$.  By the definition of $s$,  $H$ has at most $n-s$ isolated vertices in  its generalized cycle, so $\deg\det(A_x(\alpha,\beta))\leq n-s$.  Let $H_o=\p\du C$ with $C$ composed of $n-s$ isolated vertices.  Then $H_o$ is the only linear subgraph that contributes to the coefficient of $x^{n-s}$ in $\det(A_x(\alpha,\beta))$ and this coefficient is nonzero.  Therefore, $\deg\det(A_x(\alpha,\beta))=n-s$.

Let $\Delta_j(A_x)$ be the greatest common divisor of all $j\times j$ minors of $A_x$. 
Since $\Delta_{n-t}(A_x)$ is the greatest common divisor of all $(n-t)\times (n-t)$ minors of $A_x$, $\Delta_{n-t}(A_x)$ divides $\det(A_x(\alpha,\beta))$.  By \cite{KS},
$\Delta_{n-t}(A_x)$ is equal to $\Delta_{n-t}(S)$, where $S$ is the Smith Normal Form of $A_x$.  By \cite{KS}, $S=\diag(e_1(x),\dots,e_n(x))$ where $e_i(x)$ divides $e_{i+1}(x)$.  Clearly $\Delta_k(S)=\prod_{i=1}^k e_i(x)$, so $\Delta_k(A_x)=\prod_{i=1}^k e_i(x)$.   By  \cite{KS},
$\mult_A(\lam)\ge k$ if and only if $(x-\lam)$ divides $e_{n-k+1}(x)$.  Thus $q_k=\deg e_{n-k+1}(x)$, and 
\[q_{t+1} + \cdots + q_n =\deg e_{n-t}(x)+\dots+\deg e_{1}(x)= \deg\Delta_{n-t}(A_x) \leq\deg\det(A_x(\alpha,\beta))= n-s.\] 
The inequalities (i) and (ii) follow from this by the same arguments given in the proof of  Theorem \ref{thm:sharpG}.  
\end{proof}

The next example shows that an intermediate choice of $t$ for the lower bound (ii) in Theorem \ref{rslthm} may provide a better bound than either of two extreme values of $t$, both of which yield known lower bounds.
In the case $t=1$, $q(G)\ge \rsl(1)$, and $\rsl(1)$ is the
maximum number of vertices in a unique shortest path of $G$; this bound appears in \cite{AACFMN13}.  For $t=\Zsharp(G)=\Z(G)$, $\rsl(t)=n$, and $q(G)\ge \lc\frac{n}{\Z(G)}\rc$ follows from $q(G)\ge \lc\frac{n}{\M(G)}\rc$ which is immediate from the definition and  the known inequality $\M(G)\le \Z(G)$ \cite{AIM08}.

\begin{ex}\label{ex:seth} 
Let $G$ be the graph in Figure \ref{fig:seth}.  Observe that $\rsl(1)=4$, because $(1,2,3,4)$ is both the shortest path and the unique path on four vertices between $1$ and $4$, and  no path in $G$ on five vertices is the unique path on five vertices with its endpoints.  For $\alpha =\{1,6\}$ and $\beta=\{5,9\}$, $\p=\{(1,2,3,4,5), (6,7,8,9)\}$ is an $(\alpha,\beta)$-rigid shortest linkage, so $\rsl(2)\ge 9$.  By Remark \ref{rem:RSL-RL}, any rigid shortest linkage containing ten vertices has order at least $\Z(G)=3$, so $\rsl(2)= 9$.  Thus $\lc\frac{\rsl(2)}{2}\rc=\lc\frac 9 2 \rc 
= 5>4=\lc\frac{\rsl(1)}{1}\rc$ and $\lc\frac{\rsl(2)}{2}\rc= 5>4=\lc\frac {10} 3\rc=\lc\frac{\rsl(3)}{3}\rc$.

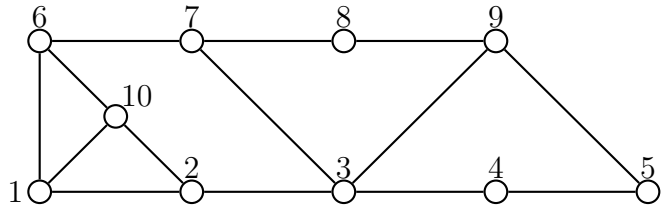
\begin{figure}[h]\vspace{-3pt}
\begin{center}
\scalebox{1}{
\begin{tikzpicture}
\node[label={[label distance=-5pt]180:$1$}] (1) at (0,0) {};
\node[label={[label distance=-5pt]$6$}] (6) at (0,2) {};
\node[label={[label distance=-5pt]$2$}] (2) at (2,0) {};
\node[label={[label distance=-5pt]$7$}] (7) at (2,2) {};
\node[label={[label distance=-5pt]$3$}] (3) at (4,0) {};
\node[label={[label distance=-5pt]$8$}] (8) at (4,2) {};
\node[label={[label distance=-5pt]$4$}] (4) at (6,0) {};
\node[label={[label distance=-5pt]$9$}] (9) at (6,2) {};
\node[label={[label distance=-5pt]$5$}] (5) at (8,0) {};
\node[label={[label distance=-5pt]45:$10$}] (10) at (1,1) {};
\draw (1)--(2)--(3)--(4)--(5);
\draw (6)--(7)--(8)--(9);
\draw (1)--(6);
\draw (1)--(10);
\draw (5)--(9)--(3)--(7);
\draw (6)--(10)--(2);
\end{tikzpicture}}\vspace{-15pt}
\end{center}\caption{\label{fig:seth} The graph $G$ in Example \ref{ex:seth}}
\end{figure}
\end{ex}


\section{Computation of eigenvalue multiplicities for families of graphs} \label{comp}

In this section we use $\rsl_G(t)$ to compute extreme possibilities for unordered multiplicity lists in the cases of several families of graphs.

\begin{prop}\label{prop:Kn}
For the complete graph $K_n$ on $n\ge 2$ vertices,\[\rsl_{K_n}(t)=t+1\text{ for }1\leq t\leq n-1.\] Moreover, there is a matrix in $\s(K_n)$ that achieves equality in every bound in Theorem $\ref{rslthm}$ simultaneously.
\end{prop}
\begin{proof}
With vertex set $\{1,2,\ldots,n\}$, there is a rigid shortest linkage with $\alpha=\{1,2,\ldots,t\}$ and $\beta=\{1,2,\ldots,t-1,t+1\}$ for $1\leq t\leq n-1$.  Clearly, no larger rigid shortest linkages can exist.  The adjacency matrix  $A\in \s(K_n)$ has spectrum $\{ (-1)^{(n-1)}, n-1 \}$ and $\sum_{i=1}^tq_i(A)=t+1=\rsl(t)$ for all $1\leq t \leq n-1$.
\end{proof}

\begin{prop}
For the cycle on $n$ vertices, $C_n$,
 \[\rsl_{C_n}(t)=\left\{\begin{array}{cl}
\lc \frac{n}{2} \rc & \text{if }t=1,  \text{ and}\\
n & \text{if }t=2.\end{array}\right. \]
Moreover, there is a matrix in $\s(C_n)$ that achieves equality in every bound in Theorem $\ref{rslthm}$ simultaneously.
\end{prop}
\begin{proof}
With vertex set $\{1,2,\ldots,n\}$ and edge set $\{\{1,2\},\{2,3\},\ldots,\{n,1\}\}$, the 
path $(1, 2, \ldots,\lc\frac{n}{2}\rc)$  is a rigid shortest linkage of order 1 with $\alpha=\{1\}$ and $\beta=\{\lc\frac{n}{2}\rc\}$.  
The paths $p_1=(1,2,\dots, n-1)$ and $p_2=(n)$ form 
a rigid shortest linkage of order 2 on $n$ vertices
with $\alpha=\{1, n\}$ and $\beta=\{n-1,n\}$.  Hence we have the stated values of $\rsl_{C_n}(t)$.

Equality in the bounds from Theorem $\ref{rslthm}$ are obtained by the adjacency matrix when $n$ is odd, and by the matrix obtained from the adjacency matrix 
by negating the $(n,1)$ and $(1,n)$ entries  when $n$ is even.
\end{proof}

\begin{prop}\label{prop:Kmn}
For the complete bipartite graph, $K_{m,n}$ with $1\leq m\leq n$ and $n\neq 1$, \[\rsl_{K_{m,n}}(t)=\left\{\begin{array}{cl}
2 & \text{if }t=1\text{ and }m\geq 2,\\
3 & \text{if }t=1\text{ and }1=m<n,  \mbox{ and }\\
t+2 & \text{if }2\leq t \leq m+n-2.\end{array}\right.\]
\end{prop}
\begin{proof}
Denote the vertices as $\{v_1,\ldots,v_m,w_1,\ldots,w_n\}$ where $v_i$ is adjacent to $w_j$ for all $i,j$. When $m=1$, relevant rigid shortest linkages are $\p_1=\{(w_1,v_1,w_2)\}$ for $t=1$, 
$\p_t=\p_1 \cup_{i=1}^{t-1}\{(u_i)\}$ for $2\leq t\leq n-1$ where the $u_i$ are any distinct vertices in $V\setminus\{v_1,w_1,w_2\}$.  

When $m\geq 2$, relevant rigid shortest linkages are 
$\p_1=\{(v_1,w_1)\}$ for $t=1$, 
$\p_2=\{(v_1,w_1),(w_2,v_2)\}$ with $\alpha=\{v_1,w_2\}$ and $\beta=\{w_1,v_2\}$ for $t=2$, 
and $\p_t= \p_2 \cup_{i=1}^{t-2}\{(u_i)\}$ for $3\leq t\leq m+n-2$ where 
the $u_i$ are any distinct vertices in $V\setminus\{v_1,v_2,w_1,w_2\}$.  

These are clearly rigid shortest linkages. No rigid shortest linkage can include more than one  non-singleton path  starting in the same set of the bipartition, nor can it include a path  with more than 3 vertices.
If a rigid shortest linkage  includes a path  with 3 vertices and central vertex $x$, then every other vertex in the part of the bipartition containing $x$ is a singleton path of the linkage (as happens vacuously above when $m=1$). 
Thus the number of vertices in a rigid shortest linkage is at most 2 more than the order of the linkage, as claimed.
\end{proof}

The family of complete bipartite graphs provides the following examples of  graphs where equality can be achieved in every bound in Theorem \ref{rslthm}, but not simultaneously.
\begin{ex}
Consider $K_{n,n}$, $n \geq 3$.   From Proposition \ref{prop:Kmn}, $\rsl(1)=2$ and
$\rsl(t)=t+2$ for $t\geq 2$.  The adjacency matrix of $K_{n,n}$ has unordered multiplicity list $(2n-2,1,1)$, which achieves equality in $\sum_{i=1}^t q_i(A)$ for $t\geq 2$, but not $t=1$. 
 In \cite{AACFMN13}, the authors proved that $q(K_{n,n})=2$ by exhibiting a matrix in $\s(K_{n,n})$ with 2 distinct eigenvalues (each of multiplicity $n$), so equality can also be achieved with a different matrix for $t=1$.
If $A\in \s(K_{n,n})$ achieves equality in Theorem \ref{rslthm} for $1\leq t\leq 2n-2$, then $A$ has unordered multiplicity list $(2n-2,2)$.  Without loss of generality, suppose that the nullity of $A$ is $2n-2$, so the rank of $A$ is 2.  Then no set of 3 columns of $A$ can be linearly independent, and from the pattern of $A$, each entry on the diagonal of $A$ is zero.  Now $\tr(A)=0$ and the sum of the eigenvalues is zero.  The remaining two eigenvalues sum to zero and are not zero, and so are not equal.  Thus, no single matrix in $\s(K_{n,n})$ can achieve all $\rsl$ bounds simultaneously.
\end{ex}

The next application is to Cartesian products. 
 The {\it Cartesian  product} of graphs $G$ and $H$ is denoted by  $G \square H$,  and has vertex set $V(G) \times V(H)$ 
with $u=(u_1,u_2)$ adjacent to $v=(v_1,v_2)$  if and only if $u_1=v_1$ and $u_2$ is adjacent to $v_2$ in $H$, or  $u_2=v_2$ and $u_1$ is adjacent to  $v_1$ in $G$.
The following lemma makes use of the fact that a shortest path in $G$ is replicated to several shortest paths in $G\square H$.
\begin{lem}\label{lem:product}
Let $G$ and $H$ be any graphs and let $t\leq |V(H)|$.  Then \[\rsl_{G\square H}(t)\geq t\cdot \rsl_G(1).\]
\end{lem}
\begin{proof}
Let $V(H)=\{v_1,\ldots,v_n\}$. Let  $\{p_1\}$ be an $(\{ a\},\{b\})$-rigid shortest linkage in $G$ on $\rsl_G(1)$ vertices.  Then 
$\alpha'=\{(a,v_i) : i=1, \ldots, t\} $ and  $\beta'=\{(b,v_i): i =1, \ldots, t \}$  are subsets of $V(G\square H)$, and the collection of the $t$ induced paths given by the vertices $V(p_1)\times\{v_i\}$ for $1\leq i \leq t$ forms an $(\alpha',\beta')$-linkage in $G\square H$ on $t\cdot \rsl_G(1)$ vertices.

To see that the linkage is a rigid shortest linkage, let $\p$ be an $(\alpha',\beta')$-linkage on $t\cdot \rsl_G(1)$ or fewer vertices.  Project the paths of $\p$ onto a single copy of $G$, so that each is an $(\{a\} ,\{b\})$-linkage in $G$.  
By hypothesis  each of these paths  has at least $\rsl_G(1)$ vertices.  Now, $\p$ is composed of $t$ of these paths, so for $\p$ to be a linkage on $t\cdot \rsl_G(1)$ or fewer vertices, each path has  $\rsl_G(1)$ vertices and lies entirely in only one copy of $G$.  Thus, $\p$ is the linkage described above because those paths are unique.
\end{proof}

The bound in Lemma \ref{lem:product} need not be tight.
\begin{ex}
Let $G$ be the graph in Example \ref{ex:seth} and let $H=P_2$.  Then $\rsl_G(1)=4$ and so $\rsl_{G\square H}(2)\geq 2\rsl_G(1)=8$ by Lemma \ref{lem:product}.  However, for this particular $G$ 
 it can be verified that  the rigid shortest linkage of order 2 on 9 vertices given in Example \ref{ex:seth} can be thought of as a rigid shortest linkage in $G\square H$ by considering the corresponding vertices in one copy of $G$.  Thus, $\rsl_{G\square H}(2) \geq 9 > 2\rsl_G(1)$. 
\end{ex}

\begin{prop}
The hypercube $Q_n$  satisfies $\rsl_{Q_n}(t)=2t$ for $1\leq t\leq 2^{n-1}$.  Moreover, there is a matrix in $\s(Q_n)$ that achieves equality in every bound in Theorem $\ref{rslthm}$ simultaneously.
\end{prop}
\begin{proof}
We have $\rsl_{Q_1}(1)=\rsl_{K_2}(1)=2$ from Example \ref{prop:Kn}.  Since hypercubes can be defined recursively by $Q_n=K_2\square Q_{n-1}$ for $n\geq 2$, Lemma \ref{lem:product} 
gives $\rsl_{Q_n}(t)\geq t\cdot \rsl_{K_2}(1)=2t$ for $1\leq t\leq 2^{n-1}=|V(Q_{n-1})|$.  Since $q(Q_n)=2$  \cite{AACFMN13} and $M(Q_n)=2^{n-1}$  \cite{AIM08}, a matrix $A\in \s(Q_n)$ satisfying $q(A)=2$ also satisfies $\sum_{i=1}^t q_i(A)=2t$ for $1\leq t\leq 2^{n-1}=M(Q_n)$, which combines with Theorem \ref{rslthm} to give the lower bound.
\end{proof}

\smallskip
\noindent
{\bf \large Acknowledgements } \\ [2.25pt]
This research began at the American Institute of Mathematics workshop {\em Zero forcing and its applications}.  We thank AIM for providing  a wonderful collaborative research environment   and for  financial support from NSF DMS 1128242.
\bibliographystyle{plain}

\end{document}